\documentclass{amsart}

\usepackage{amsfonts}
\usepackage{amssymb}
\usepackage{amsthm}
\usepackage{amsmath}
\usepackage{MnSymbol}

\usepackage{graphicx}
\usepackage{setspace}
\usepackage{multicol}
\usepackage[english]{babel}
\usepackage{fancyhdr,lastpage}

\usepackage{MnSymbol}
\usepackage{color}

\usepackage{hyperref}

\newtheorem{thm}{Theorem}[section]
\newtheorem{prop}[thm]{Proposition}
\newtheorem{cor}[thm]{Corollary}
\newtheorem{lem}[thm]{Lemma}

\newtheorem{defn}[thm]{Definition}

\DeclareMathOperator{\id}{id}

\begin{document}

\title{Some new computable structures of high rank}

%\author{M.\ Harrison-Trainor\footnote{The first author was },  G.\ Igusa\footnote{The second author was }, and J.\ F.\ Knight}

\author{Matthew Harrison-Trainor}
\address{Group in Logic and the Methodology of Science, University of California, Berkeley, USA}
\email{matthew.h-t@berkeley.edu}
\thanks{The first author was supported by NSERC PGSD3-454386-2014.}

\author{Gregory Igusa}
\address{Department of Mathematics, University of Notre Dame, USA}
\email{gigusa@nd.edu}
\thanks{The second author was supported by EMSW21-RTG-0838506.}

\author{Julia F. Knight}
\address{Department of Mathematics, University of Notre Dame, USA}
\email{j1knight@nd.edu}

\subjclass[2010]{Primary 03D45 03C57}

\date{\today}

\begin{abstract}
We give several new examples of computable structures of high Scott rank. For earlier known computable structures of Scott rank $\omega_1^{CK}$, the computable infinitary theory is $\aleph_0$-categorical. Millar and Sacks asked whether this was always the case. We answer this question by constructing an example whose computable infinitary theory has non-isomorphic countable models.

The standard known computable structures of Scott rank $\omega_1^{CK}+1$ have infinite indiscernible sequences.  We give two constructions with no indiscernible ordered triple.
%We give some new examples of computable structures of high Scott rank.  For earlier known computable structures of Scott rank $\omega_1^{CK}$, the computable infinitary theory is $\aleph_0$-categorical.  We construct an example with non-isomorphic countable models. This answers a question of Millar and Sacks.
%
%The standard known computable structures of Scott rank 
%$\omega_1^{CK}+1$ have infinite indiscernible sequences.  We construct examples with no indiscernible ordered triple.  
\end{abstract}

\maketitle

\section{Introduction}

Our main result answers an open problem posed by Millar and Sacks \cite{MS}. They asked whether every computable structure of Scott rank $\omega_1^{CK}$ is completely determined by the computable sentences it satisfies. We give a negative answer by building a computable structure of Scott rank $\omega_1^{CK}$ whose computable infinitary theory is \textit{not} $\aleph_0$-categorical. This is a new model of high Scott rank which is fundamentally different from all previously constructed models.

The Scott rank of a structure measures the internal complexity.  We give one definition below.  There are other definitions, which assign slightly different ordinals, but the important distinctions are the same for all definitions in current use.  In particular, if one definition assigns Scott rank $\omega_1^{CK}+1$, or $\omega_1^{CK}$, to a particular computable structure, then the other definitions do the same.

Let $\mathcal{A}$ be a countable structure for a computable language.  Our definition of Scott rank is based on a family of equivalence relations $\sim^\alpha$, for countable ordinals $\alpha$. Scott's original definition \cite{Sc} was based on a slightly different family of equivalence relations.  

\begin{defn}

Let $\bar{a}$ and $\bar{b}$ be tuples in $\mathcal{A}$ of the same finite length.  Then 

\begin{enumerate}

\item  $\bar{a}\sim^0\bar{b}$ if $\bar{a}$ and $\bar{b}$ satisfy the same atomic formulas

\item  for $\alpha > 0$, $\bar{a}\sim^\alpha\bar{b}$ if for each $\beta <\alpha$, for each $\bar{c}$, there exists 
$\bar{d}$, and for each $\bar{d}$, there exists $\bar{c}$, such that $\bar{a},\bar{c}\sim^\beta\bar{b},\bar{d}$.

\end{enumerate}

\end{defn}

For later use, we extend the definition $\sim^\alpha$ to allow tuples from different structures.

\begin{defn}

Let $\mathcal{A}$ and $\mathcal{B}$ be structures for the same language, and let $\bar{a}$ and $\bar{b}$ be tuples of the same length in $\mathcal{A}$, $\mathcal{B}$, respectively.  

\begin{enumerate}

\item  $(\mathcal{A},\bar{a})\sim^0 (\mathcal{B},\bar{b})$ if $\bar{a}$ and $\bar{b}$ satisfy the same atomic formulas in their respective structures.

\item  for $\alpha > 0$, $(\mathcal{A},\bar{a})\sim^\alpha (\mathcal{B},\bar{b})$ if for $\beta <\alpha$, for each $\bar{c}$ in $\mathcal{A}$, there exists $\bar{d}$ in $\mathcal{B}$, and for each $\bar{d}$ in $\mathcal{B}$, there exists $\bar{c}$ in $\mathcal{A}$, such that $(\mathcal{A},\bar{a},\bar{c})\sim^\beta (\mathcal{B},\bar{b},\bar{d})$. 

\end{enumerate}

\end{defn}

\noindent
\textbf{Remark}:  If $(\mathcal{A},\bar{a})\sim^\alpha (\mathcal{B},\bar{b})$, then for any $\Sigma_\alpha$ formula $\varphi(\bar{x})$ of $L_{\omega_1\omega}$, $\mathcal{A}\models\varphi(\bar{a})$ iff $\mathcal{B}\models\varphi(\bar{b})$.  

\bigskip

We define Scott rank, first for a tuple in a structure $\mathcal{A}$, and then for the structure itself.  

\begin{defn}\

\begin{enumerate}

\item  The \emph{Scott rank of a tuple} $\bar{a}$ is the least $\alpha$ such that for all $\bar{b}$, if 
$\bar{a}\sim^\alpha\bar{b}$, then for all $\gamma > \alpha$, $\bar{a}\sim^\gamma\bar{b}$.

\item  The \emph{Scott rank of the structure} $\mathcal{A}$ is the least ordinal greater than the Scott ranks of all tuples in $\mathcal{A}$. 

\end{enumerate}

\end{defn}   

Nadel \cite{N} observed that for a computable structure $\mathcal{A}$, two tuples are automorphic just in case they satisfy the same computable infinitary formulas.  This implies that the Scott rank of $\mathcal{A}$ is at most 
$\omega_1^{CK} + 1$.  The following is well-known.    

\bigskip
\noindent
\textbf{Fact}.  Let $\mathcal{A}$ be a computable structure.  

\begin{enumerate}

\item  $\mathcal{A}$ has computable Scott rank iff there is a computable ordinal $\alpha$ such that for all tuples $\bar{a}$ in $\mathcal{A}$, the orbit of $\bar{a}$ is defined by a computable $\Sigma_\alpha$ formula.  

\item  $\mathcal{A}$ has Scott rank $\omega_1^{CK}$ iff for each tuple $\bar{a}$, the orbit is defined by a computable infinitary formula, but for each computable ordinal $\alpha$, there is a tuple $\bar{a}$ whose orbit is not defined by a computable $\Sigma_\alpha$ formula.  

\item  $\mathcal{A}$ has Scott rank $\omega_1^{CK}+1$ iff there is a tuple $\bar{a}$ whose orbit is not defined by a computable infinitary formula.  

\end{enumerate}

There are familiar examples of computable structures having various computable ordinal ranks.  The canonical example of a computable structure of Scott rank $\omega_1^{CK}+1$ is the \emph{Harrison ordering}, a linear order with order type $\omega_1^{CK}(1+\eta)$ (see \cite{Harrison}).  Producing a computable structure of Scott rank $\omega_1^{CK}$ took longer.  Makkai gave an example of an arithmetical structure of Scott rank $\omega_1^{CK}$ \cite{M}, a ``group-tree''.  In \cite{KM}, Makkai's construction is re-worked to give a computable structure.  In \cite{CKM}, there is a simpler example, a computable tree of Scott rank 
$\omega_1^{CK}$.  In \cite{CGKM}, the tree is used to produce further structures in familiar classes---a field, a group, etc.\footnote{Although \cite{KM} was not published until 2011, it was written before \cite{CKM}, which was published in 2006, and \cite{CGKM}, which was published in 2009.}  

These examples of computable structures of Scott rank $\omega_1^{CK}$ all have the feature that the computable infinitary theory is $\aleph_0$-categorical.  The conjunction of the computable infinitary theory forms a Scott sentence.  In \cite{MS}, Millar and Sacks produced a structure $\mathcal{A}$ of Scott rank 
$\omega_1^{CK}$ such that the computable infinitary theory of $\mathcal{A}$ is not $\aleph_0$-categorical.  The structure $\mathcal{A}$ is not computable; it is not even hyperarithmetical, but it has the feature that $\omega_1^\mathcal{A} = \omega_1^{CK}$.  This means that $\mathcal{A}$ lives in a fattening of the admissible set $L_{\omega_1^{CK}}$.  

In \cite{MS}, Millar and Sacks asked whether there is a computable structure of Scott rank $\omega_1^{CK}$ whose computable infinitary theory is not $\aleph_0$-categorical.  The question is asked again in \cite{CGKM}. Millar and Sacks also asked whether there are similar examples for other countable admissible ordinals.  In \cite{F}, Freer proved the analog of the result of Millar and Sacks, producing, for an arbitrary countable admissible ordinal $\alpha$, a structure $\mathcal{A}$ with $\omega_1^{\mathcal{A}} = \alpha$, such that the theory of $\mathcal{A}$ in the admissible fragment $L_\alpha$ is not $\aleph_0$-categorical.  Freer's structure is not in 
$L_\alpha$, but in a fattening of $L_\alpha$.  The main result of the present paper says that there is a computable structure of Scott rank $\omega_1^{CK}$ for which the conjunction of the computable infinitary theory is not a Scott sentence. This answers positively the question of Millar and Sacks mentioned above. The construction appears in Section 2.

By an ``indiscernible sequence'', we mean a infinite sequence that is indiscernible for $L_{\omega_1 \omega}$ formulas.

\begin{defn}

Fix a structure $\mathcal{A}$. An \emph{indiscernible sequence} in $\mathcal{A}$ is a sequence $(a_i)_{i \in \omega}$ of elements of $\mathcal{A}$ such that for any two finite subsequences $a_{i_1},\ldots,a_{i_n}$ and $b_{j_1},\ldots,b_{j_n}$ (with $i_1 < i_2 < \dots < i_n$ and $j_1 < j_2 < \dots < j_n$) satisfy the same $L_{\omega_1 \omega}$ formulas.

\end{defn}

The Harrison ordering obviously has an infinite indiscernible sequence.  Other examples of computable structures of Scott rank $\omega_1^{CK} + 1$ share this feature.  Goncharov and Knight (unpublished) asked whether every computable structure of Scott rank $\omega_1^{CK}+1$ has an infinite indiscernible sequence.  At the same time, they also noticed that the structures of Scott rank $\omega_1^{CK}$ constructed by Makkai \cite{M}, and Knight and Millar \cite{KM}, did not have an infinite indiscernible sequence. 

In Section 3, we describe two constructions producing computable structures of Scott rank $\omega_1^{CK} + 1$ with not even an indiscernible ordered triple.  The first is produced by taking a Fra\"{i}ss\'{e} limit with infinitely many infinite equivalence classes, and putting the structure of the Harrison ordering on the equivalence classes. Alhough this structure has no indiscernible triples, it is effectively bi-interpretable \cite{HMMM} with the Harrison ordering, and hence it has an infinite indiscernible sequence of imaginaries.

The second example is a modified version of Makkai's construction \cite{M}. %JKGiven a tree $T$, Makkai constructed a ``group-tree'' $\mathcal{A}(T)$, 
Makkai \cite{M} gave a ``computable operator'' taking an input tree $T$ to a group-tree $\mathcal{A}(T)$---the group-tree $\mathcal{A}(T)$ is computable uniformly in the input tree $T$. The structure $\mathcal{A}(T)$ is built by putting a group structure on each level of the tree (the language of the structure $\mathcal{A}(T)$ does not include the group operation, but, instead, has a collection of unary functions).  Makkai constructed a $\Delta^0_2$ ``thin'' tree $T$ such that $\mathcal{A}(T)$ had Scott rank $\omega_1^{CK}$. Knight and Millar \cite{KM} modified the construction to make the input tree (and hence the output group-tree) computable.  All of the elements of $\mathcal{A}(T)$ are definable from a collection of parameters $g_n$, one at each level.  Hence, $\mathcal{A}(T)$ does not have an indiscernible ordered triple.  We will show that if our input tree $T$ is the sequence of descending sequences in the Harrison ordering, then the resulting group-tree $\mathcal{A}(T)$ has Scott rank 
$\omega_1^{CK}+1$, but it still does not have an indiscernible triple (or even an indiscernible triple of imaginaries).

\section{\texorpdfstring{Scott rank $\omega_1^{CK}$}{Scott rank w1CK}}  

In this section, our goal is to prove the following. 

\begin{thm}

There is a computable structure $\mathcal{M}$ with Scott rank $\omega_1^{CK}$ such that the computable infinitary theory of 
$\mathcal{M}$ is not $\aleph_0$-categorical.

\end{thm}

We will use some material on trees from \cite{CKM}.  The trees are isomorphic to subtrees of $\omega^{<\omega}$.  We use a language with a successor relation.  Here are the facts that we need.  There is a computable tree $T^*$ of Scott rank $\omega_1^{CK}$.  In addition, for a fixed $\Pi^1_1$ path $P$ through $\mathcal{O}$, there is a family of approximating trees $(T^a)_{a\in P}$ of computable Scott rank.  For $a\in P$ such that $|a| = \alpha$, $T^a$ has tree rank at most $\omega(\alpha+1)$, and $T^*\sim^\alpha T^a$.  The family $(T^a)_{a \in P}$ is computable uniformly in $a$, and the tree ranks of the nodes of $T^a$ are also computable, uniformly in $a$, in the sense that we can effectively label the nodes of $T^a$ by pairs $(b,n)$, where $\sigma\in T^a$ has label $(b,n)$ for $b\in P$ just in case $\sigma$ has tree rank $\omega\cdot\beta + n$ for $|b| = \beta$.  The tree $T^*$ and the approximations $T^a$ are all ``rank-homogeneous''.  

\begin{defn}

$T$ is \emph{rank-homogeneous} provided that for each node $x$ at level $n$, if $x$ has tree rank $\alpha$ and there is a node at level $n+1$ (not necessarily a successor of $x$) of tree rank $\beta < \alpha$, then $x$ has infinitely many successors of tree rank $\beta$.  Also, for each node $x$ at level $n$, if $x$ has infinite rank, then it has infinitely many successors of infinite rank, in addition to infinitely many of each ordinal rank $\beta$ that occurs at level $n+1$.  

\end{defn}

\begin{defn}

For $\mathcal{A}$ and $\mathcal{B}$ rank-homogeneous trees, with $\bar{a}$ in $\mathcal{A}$ and $\bar{b}$ in $\mathcal{B}$, we write $(\mathcal{A},\bar{a})\approx^\alpha(\mathcal{B},\bar{b})$ provided that 

\begin{enumerate}

\item  for all $n$, the tree ranks less than $\omega\alpha$ of nodes at level $n$ are the same in $\mathcal{A}$ and $\mathcal{B}$,

\item  the subtree of $\mathcal{A}$ ``generated'' by $\bar{a}$ (by closing under predecessors) is isomorphic to the subtree of $\mathcal{B}$ generated by $\bar{b}$, with an isomorphism taking $\bar{a}$ to $\bar{b}$,

\item for corresponding elements $x$ in the subtree of $\mathcal{A}$ generated by $\bar{a}$ and $x'$ in the subtree of $\mathcal{B}$ generated by $\bar{b}$, either the tree ranks of $x$ and $x'$ match, or else both are at least $\omega\cdot\alpha$.

\end{enumerate}  

\end{defn}

\begin{lem}

Let $\mathcal{A}$ and $\mathcal{B}$ each be one of our trees $T^{*}$ or $T^a$.  If \\
$(\mathcal{A},\bar{a})\approx^\alpha (\mathcal{B},\bar{b})$, then $(\mathcal{A},\bar{a})\sim^\alpha (\mathcal{B},\bar{b})$.  

\end{lem}  

\begin{proof}

The statement is clear for $\alpha = 0$.  Also, if $\alpha$ is a limit ordinal, and the statement holds for all $\beta <\alpha$, then it holds for $\alpha$.  Supposing that it holds for $\alpha$, we prove it for $\alpha+1$.  For simplicity, suppose that $\bar{a}$ and $\bar{b}$ are subtrees.  Let $a$ be an element of $\bar{a}$ that has a new successor $c$ at the top of a finite subtree $\bar{c}$.  Let $b$ be the element corresponding to $a$.  We need $d$ and $\bar{d}$ matching $c$ and $\bar{c}$.  If the tree ranks of $a$ and $b$ match, and are less than $\omega(\alpha+1)$, then we can choose $d$ and $\bar{d}$ with tree ranks matching the corresponding elements of $c$ and $\bar{c}$.  If the tree ranks of $a$ and $b$ are at least $\omega(\alpha+1)$, and the tree rank of $c$ and the elements of $\bar{c}$ are at least $\omega\cdot\alpha$, then we choose $d$ and $\bar{d}$ also with tree ranks at least $\omega\cdot\alpha$.  We can choose $d$ of rank $\omega\cdot\alpha + n$ for $n$ sufficiently large to leave room for choosing the rest of $\bar{d}$.  If some of the elements of $\bar{c}$ have tree ranks less than $\omega\cdot\alpha$, then we choose the corresponding elements of $\bar{d}$ with matching tree ranks.  
\end{proof}

Given $a$, we can effectively find a Scott sentence for $T^a$.  We give the tree rank of the top node, and for each $n$, we say what are the tree ranks of nodes at level $n$.  Finally, we say that the tree is ``rank-homogeneous''; i.e., for all $x$ at level $n$ of tree rank $\beta$, and all $\gamma < \beta$ such that there is a node of tree rank $\gamma$ at level $n+1$, $x$ has infinitely many successors of tree rank $\gamma$. This is effective since we have, uniformly in $a$, a function giving the ranks of the nodes of $T^a$.

\bigskip

We want a computable copy $(U,<_U,S_U)$ of the Harrison ordering with the successor relation, with a family of trees 
$(T_u)_{u\in U}$, uniformly computable in $u$, such that if $pred(u)$ has order type $\alpha$ with notation $a\in P$, then $T_u\cong T^a$, and if $pred(u)$ is not well-ordered, then $T_u\cong T^{*}$. 

\begin{lem}

There is a computable structure $\mathcal{A}$ with universe the union of disjoint sets $U$ and $V$, with an ordering $<$ of type 
$\omega_1^{CK}(1+\eta)$ and successor relation $S$ on $U$, and with a function $Q$ from $V$ to $U$ such that for each $u\in U$, $Q^{-1}(u)$ is an infinite set, with a tree structure $T_u$.  If $pred(u)$ has order type $\alpha$ with notation $a\in P$, then $T_u\cong T^a$.  If $pred(u)$ is not well ordered, then $T_u\cong T^*$.  %Moreover, there is a function $tr$ from $V$ to $U\times\omega$ such that if $x$ is a node in $T_u$, then 
%$x$ has ordinal tree rank $\omega\beta+n$ just in case $tr(x) = (b,n)$, where $b$ is the element of $U$ with $pred(b)$ of type $\beta$.    

\end{lem}

\begin{proof}

To prove the lemma, we use Barwise-Kreisel Compactness.  Let $\Gamma$ be a $\Pi^1_1$ set of computable infinitary sentences %JK
%characterizing computable structures
in the language of $\mathcal{A}$ %,JK such that 
saying that $U$ and $V$ are disjoint, $U$ is linearly ordered by $<_U$, $S_U$ is the successor relation on the ordering, $Q$ maps $V$ onto $U$ such that for each $u\in U$, $Q^{-1}(u)$ is infinite, $S_T$ is the union of successor relations putting a tree structure $T_u$ on the set $Q^{-1}(u)$, %and $tr$ is a function from $V$ to $U\times\omega$, 
with further axioms guaranteeing the following:   

\begin{enumerate}

\item  for each computable ordinal $\alpha$, the ordering $(U,<_U)$ has an initial segment of type $\alpha$, 

\item  for each computable ordinal $\alpha$, each $u\in U$ %JK lies on 
is the left endpoint of an interval of type $\omega^\alpha$,

\item  the ordering $(U,<_U)$ has no infinite hyperarithmetical decreasing sequence,

\item  for each $u\in U$, if $pred(u)$ has order type $\alpha$, where $a\in P$ is the notation for $\alpha$, then $T_u\cong T^a$, 

\item  for a computable ordinal $\alpha$, if $u <_U v$, where $pred(u)$ has order type $\alpha$, then $T_u$ and $T_v$ satisfy the same computable $\Sigma_\alpha$ sentences.

%\item  if $u\in U$, where $pred(u)$ has order type $\alpha$ and $n\in\omega$, then $tr(x) = (u,n)$ just in case $x$ has tree rank 
%$\omega\alpha + n$.     

\end{enumerate}

For a hyperarithmetical set $\Gamma'\subseteq\Gamma$, there is a computable ordinal $\gamma$ bounding the ordinals $\alpha$ corresponding to sentences in $\Gamma'$ of types (1), (2), (4), and (5).  Then we get a model of $\Gamma'$ as follows.  We fix computable sets $U$ and $V$ in advance.  Let $c$ be the notation for $\gamma$ in $P$.  Since $<_O-pred(c)$ is c.e., we have a computable function $f$ from $U$ $1-1$ onto $<_O-pred(c)$.  For $x,y\in U$, $x <_U y$ iff $f(x) <_O f(y)$.  Let $S_U(x,y)$ iff $f(y) = 2^{f(x)}$.  If $f(x) = a$, then $T_x\cong T^a$.  %, and the values of the function 
%$tr(x)$ correctly represent the tree ranks of the nodes of the trees $T_u$.  
Since every hyperarithmetical subset of $\Gamma$ has a model, the whole set does.  In this model, $(U,<_U)$ has order type $\omega_1^{CK}(1+\eta)$.  For $u\in U$ such that $pred(u)$ is not well-ordered, $T_u$ satisfies the computable infinitary sentences true in $T^*$.  Since $T_u$ is computable, it must be isomorphic to $T^*$.   
\end{proof}

%Now that we have this lemma, we identify $P$ with the initial segment of $U$ of order type $\omega_1^{CK}$, and for $u = a$ in this initial segment, we write $T^a$ for $T_u$.

%%%?

\begin{lem}

Let $I$ be the well-ordered initial segment of $U$ of order type $\omega_1^{CK}$.  There is a uniformly computable sequence 
$(R_n)_{n \in \omega}$ of infinite subsets of $U$ with the following properties:

\begin{enumerate}

\item  $R_0$ contains some element of $I$,

\item  for each $n$, there exists $u\in I$ that is an upper bound on $R_n\cap I$,

\item  for each $u\in R_n$, there exists $v\in R_{n+1}$ such that $u < v$,

\item  for each $u\in R_n\cap I$, there exists $v\in R_{n+1}\cap I$ such that $u < v$, 

\item  $\cup_n R_n$ is unbounded in the well-ordered initial segment of $U$.

\end{enumerate}

\end{lem}

\begin{proof}

Fix $u_0\notin I$, $u_1\in I$.  Let $R_0$ consist of all elements of $u \geq u_0$, plus $u_1$.  Given $R_m$ for $m\leq n$, let 
$R_{n+1}$ consist of the successors in $U$ of each $u\in R_n$, plus the $\omega$-first element of $U$ not in 
$\cup_{m\leq n} R_m$.    
\end{proof}

Let $T$ be the tree of finite sequences $(u_1,\ldots,u_n)$ that are increasing in $(U,<_U)$, with $u_n \in R_n$.  We define a function $H$ that takes each non-empty sequence $\sigma\in T$ to its last term $u_n$.  This function is obviously computable.  We are about to describe the structure $\mathcal{M}$.  The language includes the following:

\begin{enumerate}

\item  unary predicates $A$ and $B$---these will be disjoint,

\item  for each $\tau\in T$, a binary relation $C_\tau$ that associates to each $x \in A$ a subset $T_{(\tau,x)}$ of $B$, where for distinct pairs 
$(\tau,x)$ and $(\tau',x')$, the sets $T_{(\tau,x)}$ and $T_{(\tau',x')}$ are disjoint, 

\item  a binary successor relation $S$ that puts a tree structure on each set $T_{(\tau,x)}$.

\end{enumerate}         

For the structure $\mathcal{M}$, we let $A^\mathcal{M}$ consist of (codes for) the elements of $T$.  For each $\sigma\in A$ and 
$\tau\in T$, we define $S^\mathcal{M}$ so that
\[T_{(\tau,\sigma)}
\cong
\left\{
\begin{array}{ll}
T_u & \mbox{if $\tau\preceq\sigma\ \&\ H(\tau) = u$}\\
T^* & \mbox{if $\tau\not\preceq\sigma$}
\end{array}\right. \]

The structure $\mathcal{M}$ is computable.  We note that for $u\in T$, if $pred(u)$ is well ordered, then $T_u$ is isomorphic to the appropriate $T^a$, while if $pred(u)$ is not well ordered, then $T_u$ is isomorphic to $T^*$.  

\bigskip

To show that the computable infinitary theory of $\mathcal{M}$ is not $\aleph_0$-categorical, we produce a second model 
$\mathcal{N}$, not isomorphic to $\mathcal{M}$.  We want a path through $T$ with special features.

\begin{lem}

There is a path $\pi$ through $T$ such that for all $n$, $\pi(n)\in R_n$, and $ran(\pi)$ is co-final in $I$.  

\end{lem}

\begin{proof}

Let $(u_n)_{n\in\omega}$ be a list of the elements of $I$.  Let $\pi(0)\in R_0\cap I$.  Given $\pi(n)$, take the first $k$ such that $u_k > \pi(n)$.  We choose $\pi(n+1)$ to be some $v > \pi(n)$ in $R_{n+1}\cap I$.  If possible, we take $v\geq u_k$.  Since $I$ is co-final in $\cup_n R_n$, there is some $m$ such that $R_m\cap I$ has an element $v\geq u_k$, and then the same is true for all $m'\geq m$.  So, for each $k$, we will come to $m$ such that we can choose $\pi(m)\geq u_k$.  
\end{proof}

Let $\mathcal{N}$ be the extension of $\mathcal{M}$ with an additional element of $A^\mathcal{N}$ representing the path $\pi$.  We define $S^\mathcal{N}$ on $T_{(\tau,\pi)}$ so that  
\[T_{(\tau,\pi)}\cong \left\{\begin{array}{ll}
T_u & \mbox{if $H(\tau) = u$ and $\tau\preceq \pi$}\\
T^* & \mbox{otherwise}
\end{array}\right.\]

\begin{lem}

$\mathcal{M}$ and $\mathcal{N}$ are not isomorphic.  

\end{lem}

\begin{proof}
For a fixed $\sigma\in T$, there are only finitely many $\tau$ such that $\tau\preceq\sigma$. In $\mathcal{M}$, for a fixed $\sigma$, all but finitely many of the trees $T_{(\tau,\sigma)}$ are isomorphic to $T^*$. On the other hand, in $\mathcal{N}$, there are infinitely many initial segments $\tau$ of $\pi$ with $T_{(\tau,\pi)}$ isomorphic to some $T_u \ncong T^*$. Thus, no element of $\mathcal{M}$ can be mapped isomorphically to $\pi \in \mathcal{N}$.
\end{proof} 

\begin{defn}

We write $\mathcal{A}\preceq_{\infty} \mathcal{B}$ if for any computable infinitary formula $\varphi(\bar{x})$ and any $\bar{a}$ in $\mathcal{A}$, $\mathcal{A}\models\varphi(\bar{a})$ iff $\mathcal{B}\models\varphi(\bar{a})$.

\end{defn}
 
To show that $\mathcal{N}$ satisfies the computable infinitary theory of $\mathcal{M}$, we show that $\mathcal{M}\preceq_{\infty}\mathcal{N}$.  For this, it is enough to show that for any computable ordinal $\alpha$ and any tuple $\bar{a} \in \mathcal{M}$, $(\mathcal{M},\bar{a}) \sim^\alpha (\mathcal{N},\bar{a})$.

\begin{lem}

Let $\mathcal{A}$ and $\mathcal{B}$ be structures, each isomorphic to one of $\mathcal{M}$ or $\mathcal{N}$.  Let $\bar{a} = (a_1,\ldots,a_n)$ be a tuple in 
$\mathcal{A}$, and let $\bar{b} = (b_1,\ldots,b_n)$ be a tuple in $\mathcal{B}$ of the same length. Suppose that:

\begin{enumerate}
\item $\bar{a}$ and $\bar{b}$ satisfy the same atomic formulas, 

\item for each $a_i$ and the corresponding $b_i$ in the predicate $A$ (and with $u$ being the element of $U$ with $pred(u)$ having order type $\alpha$), for each $n$, if one of $a_i(n)$ or $b_i(n)$ is defined and $\leq_U u$, then $a_i(n) = b_i(n)$, and

\item for each $a_i$ and corresponding $b_i$, both in $A$, and for each $\tau \in T$, we have \linebreak
$(T_{(\tau,a_i)},\bar{c}) \sim^\alpha (T_{(\tau,b_i)},\bar{d})$, where $\bar{c}$ consists of the elements from $\bar{a}$ that are in $T_{(\tau,a_i)}$, and $\bar{d}$ consists of the corresponding elements from $\bar{b}$.  

\end{enumerate}
(We assume that for each element $a$ of the tuple $\bar{a}$ that is in the predicate $B$, the corresponding element $a'$ of the predicate $A$ with $a \in T_{(\tau,a')}$ is also present in the tuple $\bar{a}$. We make a similar assumption about $\bar{b}$.) Then $(\mathcal{A},\bar{a}) \sim^\alpha (\mathcal{B},\bar{b})$.

\end{lem}

Note that if $\bar{a}$ and $\bar{b}$ both consist solely of elements from the predicate $A$, then (1) and (2) imply (3).

\begin{proof}

We argue by induction on $\alpha$. Suppose that $\bar{a}$ and $\bar{b}$ satisfy the conditions above for $\alpha$. Given $\beta < \alpha$ and $\bar{a}'$ a tuple in $\mathcal{A}$, we will find a tuple $\bar{b}'$ in $\mathcal{B}$ such that $\bar{a},\bar{a}'$ and $\bar{b},\bar{b}'$ satisfy the conditions above for $\beta$. It suffices to assume about $\bar{a}'$ that for each element $a$ of the tuple $\bar{a}'$ that is in the predicate $B$, the corresponding element $a'$ of the predicate $A$ with $a \in T_{(\tau,a')}$ is present in the tuple $\bar{a},\bar{a}'$.

First, for each $a_i' \in A$, we choose $b_i'$ such that for all $\tau \in T$, $T_{(\tau,a_i')} \sim^\alpha T_{(\tau,b_i')}$. If $a_i' \in T$, then we choose $b_i' = a_i'$, if possible.  However, it may be that $a_i'$ is already in $\bar{b}$.  So, instead, let $n$ be the length of $a_i'$, and choose $b_i' = a_i' \string^ \langle v \rangle$ for some sufficiently large $v \in R_n$ with $pred(v)$ ill-founded. For each 
$\tau$, if $\tau \preceq a_i'$, then $\tau \preceq b_i'$, so that $T_{(\tau,a_i')} \cong T_u \cong T_{(\tau,b_i')}$ for some $u$.  If $\tau \npreceq a_i'$, then either $\tau \npreceq b_i'$ or $\tau = b_i'$ (whence $H(\tau) = v$).  Either way, $T_{(\tau,a_i')} \cong T^* \cong T_{(\tau,b_i')}$ by choice of $v$.

If, instead, $a_i' = \pi$, let $u \in U$ be such that $pred(u)$ is well-founded with order type $\alpha$. Let $\sigma$ be the initial segment of $\pi$ consisting of all of the entries $v$ of $\pi$ with $v \leq_U u$. Let $b_i'$ be a code for $\sigma \string^ \langle v \rangle$ for some $v \in R_n$ with $pred(v)$ ill-founded, which is sufficiently large that $b_i'$ codes a new element. Then for all $\tau \preceq \sigma$, we have $\tau \preceq \pi$ and so $T_{(\tau,a_i')} \cong T_{(\tau,b_i')}$. For all $\tau \npreceq \sigma$, either $\tau \npreceq \pi$, in which case $T_{(\tau,a_i')} \cong T^* \cong T_{(\tau,b_i')}$, or $\tau \preceq \pi$, in which case, $T_{(\tau,a_i')} \cong T_v \sim^\alpha T^* \cong T_{(\tau,b_i')}$ for some $v >_U u$. In this manner, we may reduce to the case where $\bar{a}_i'$ contains only elements of the predicate $B$. 

For each element $a$ from $\bar{a}$ in the predicate $A$, let $b$ be the corresponding element from $\bar{b}$. Fix $\tau \in T$. Let $\bar{c}$ consist of the elements from $\bar{a}$, and let $\bar{c}'$ consist of the elements from $\bar{a}'$ that are in $T_{(\tau,a)}$.  Similarly, let $\bar{d}$ consist of the elements from $\bar{b}$ that are in $T_{(\tau,b)}$. By assumption, $(T_{(\tau,a)},\bar{c}) \sim^\alpha (T_{(\tau,b)},\bar{d})$. Thus, there is $\bar{d}'$ such that $(T_{(\tau,a)},\bar{c},\bar{c}') \sim^\beta (T_{(\tau,b)},\bar{d},\bar{d}')$. The tuple $\bar{b}'$ consists of the elements of the tuples $\bar{d}'$ for each $a$ and $\tau$.
\end{proof}

\begin{lem}

$\mathcal{M}\preceq_\infty \mathcal{N}$  

\end{lem}

\begin{proof}

Let $\bar{a}$ be a tuple in $\mathcal{M}$. Then by the previous lemma, $(\mathcal{M},\bar{a})\sim^\alpha (\mathcal{N},\bar{a})$.  It follows that for any $\Sigma_\alpha$ formula $\varphi(\bar{x})$, if $\mathcal{M}\models\varphi(\bar{a})$, then $\mathcal{N}\models\varphi(\bar{a})$.  This proves the lemma.          
\end{proof}

\begin{lem}

$\mathcal{M}$ has Scott rank $\omega_1^{CK}$.

\end{lem}

\begin{proof} 

First, note that there is an automorphism of $\mathcal{M}$ taking $\sigma_1 \in A$ to $\sigma_2 \in A$ if and only if for each $\tau \in T$, $T_{(\tau,\sigma_1)}$ is isomorphic to $T_{(\tau,\sigma_2)}$. This is the case if and only if for each $n$, if $pred(\sigma_1(n))$ or $pred(\sigma_2(n))$ is well-founded, then $\sigma_1(n) = \sigma_2(n)$.

Fix $\sigma \in A$.  We define the orbit of $\sigma$ by saying, for the finitely many $\tau \preceq \sigma$ with $H(\tau) = u$ and $pred(u)$ well-founded, that $T_{(\tau,\sigma)}$ is isomorphic to $T_u$, and for each other $\tau$, that $T_{(\tau,\sigma)}\cong T^*$. We can express the former by a computable formula using the Scott sentences of the $T_u$. For the latter, note that it suffices to say that $T_{(\tau,\sigma)}$ is isomorphic to $T^*$ for only those $\tau$ of length at most $n+1$, where $n$ is the length of $\sigma$.  We cannot express this directly by a computable infinitary formula, but there is a computable infinitary formula that is satisfied in $\mathcal{M}$ exactly by such elements of $A$. Let $\alpha < \omega_1^{CK}$ be large enough that for all elements $v$ of $R_0,\ldots,R_n$ with $pred(v)$ well-founded (recalling that such $v$ are not cofinal in the initial segment of $U$ of order type $\omega_1^{CK}$), $T_v \nsim^\alpha T^*$. Then, for $\tau$ of length at most $n+1$, $T_{(\tau,\sigma)}$ is isomorphic to $T^*$ if and only if $T_{(\tau,\sigma)} \sim^\alpha T^*$. This can be expressed by a computable infinitary formula.

The Scott rank of a tuple $\bar{b}$ in $T_{(\tau,\sigma)}$ is not greater than the Scott rank of $\bar{b}$ in $\mathcal{M}$.  Therefore, the Scott rank of $\mathcal{M}$ is at least $\omega_1^{CK}$, since there are many $\tau$ and $\sigma$ such that $T_{\tau,\sigma}$ is isomorphic to $T^*$, which has Scott rank $\omega_1^{CK}$.
The Scott rank of $\mathcal{M}$ is at most $\omega_1^{CK}$, since we can define the orbit of any tuple by a computable infinitary formula. For a tuple $\bar{u},\bar{v}$ in $\mathcal{M}$, where $\bar{u}\in A$ and $\bar{v}\in B$, we can define the orbit as follows: for each element $\sigma\in A$ in the tuple $\bar{u}$, we give a definition as above, and if $\bar{b}$ is the part of the tuple $\bar{v}$ in a particular tree $T_{(\tau,\sigma)}$, then we say what is the orbit of $\bar{b}$ in $T_{(\tau,\sigma)}$. Here we use the fact that each of the trees $T_{(\tau,\sigma)}$ itself has Scott rank at most $\omega_1^{CK}$.     
\end{proof}

\section{\texorpdfstring{Scott rank $\omega_1^{CK}+1$}{Scott rank w1CK + 1}}

We begin this section by proving the following:

\begin{thm}
\label{thm:indis-triple}

There is a computable structure of Scott rank $\omega_1^{CK}  + 1$ with no indiscernible ordered triple.

\end{thm}

Our structure will be a \emph{Fra\"{i}ss\'{e}} limit, obtained from a class $K$ of finite structures satisfying the hereditary, amalgamation, and joint embedding properties, abbreviated $HP$, $AP$, and $JEP$.  For a discussion of Fra\"{i}ss\'{e} limits from the point of view of computability, see \cite{CHMM}.  We note that Henson \cite{H} gave an example of a homogeneous triangle-free graph.    

\begin{proof}

We define a class $K$ of finite structures with signature consisting of binary relations $E$ and $(C_i)_{i\in\omega}$.  We view the relations $C_i$ as ``colors'' (in the sense of Ramsey theory) with which we color (unordered) pairs of vertices. A finite structure $\mathcal{A}$ will be in $K$ if 
$E$ is an equivalence relation, the $C_i$ color the unordered pairs of vertices (i.e., $xC_iy$ if and only if $yC_ix$) with exactly one color per edge, and there are no \emph{monochromatic triangles}; i.e., there is no $i \in\omega$ and $x, y, z \in\mathcal{A}$ with $xC_iyC_izC_ix$.

\bigskip
\noindent
\textbf{Claim 1}. $K$ satisfies the hereditary property ($HP$), amalgamation property ($AP$), and the joint embedding property ($JEP$).

\begin{proof} [Proof of Claim 1]

The $HP$ is clear. To see that $K$ has the $AP$, suppose that $\mathcal{A} \subseteq \mathcal{B}, \mathcal{C}$ are structures in $K$. We define a structure $\mathcal{D}\in K$ extending $\mathcal{B}$ and $\mathcal{C}$. We can extend the equivalence relation to $\mathcal{D}$. Since there are only finitely many elements of $\mathcal{A}$ and $\mathcal{B}$, only finitely many colors have been used so far. To color edges $(x, y)$, where $x \in \mathcal{B} - \mathcal{A}$ and $y \in\mathcal{C}-\mathcal{A}$, simply choose $i$ that has not colored any edge yet, and color $(x,y)$ with $i$.  This cannot introduce any monochromatic triangles.  A similar argument, omitting $\mathcal{A}$, shows that $K$ has the $JEP$. 
\end{proof}

Note that we can effectively list the structures of $K$. Thus, $K$ has a computable Fra\"{i}ss\'{e} limit $\mathcal{M}$, and $\mathcal{M}$ has, as its Scott sentence, a computable infinitary sentence $\varphi$. Models of $\varphi$ have infinitely many equivalence classes, all of which have infinitely many elements. Now, let $\mathcal{N}$ be an expansion of $\mathcal{M}$ with a linear order $\leq$ of order type $\omega_1^{CK}(1+\eta)$ on the equivalence classes. The structure $\mathcal{N}$ has a computable copy, since we can find an effective labeling of the equivalence classes by elements of $\omega$, and use the Harrison ordering. Let $R:\mathcal{A}\rightarrow \omega_1^{CK}(1+\eta)$ be the resulting effective order-preserving map, which respects equivalence classes, and which induces a bijection between the equivalence classes and elements of $\omega_1^{CK}(1+\eta)$. 

\bigskip
\noindent
\textbf{Claim 2}. For $\alpha\geq 1$ and tuples $\bar{x}, \bar{y}$ in $\mathcal{N}$, $\bar{x}\sim^\alpha\bar{y}$ in 
$\mathcal{N}$ if and only if $R(\bar{x}) \sim^\alpha R(\bar{y})$ in $\omega_1^{CK}(1+\eta)$ and $\bar{x}\equiv_{at}\bar{y}$, where $\bar{x}\equiv_{at}\bar{y}$ if $\bar{x}$ and $\bar{y}$ satisfy the same atomic formulas in $\mathcal{N}$.

\begin{proof} [Proof of Claim 2]

Suppose that $R(\bar{x})\sim^\alpha R(\bar{y})$ and $\bar{x}\equiv_{at} \bar{y}$. Then we will show that $\bar{x}\sim^\alpha\bar{y}$. Take $\beta <\alpha$ and $\bar{x}'$ a new tuple of elements. Then there is $\bar{a}$ in $\omega_1^{CK}(1 + \eta)$ such that $R(\bar{x})R(\bar{x}') \sim^\beta R(\bar{y})\bar{a}$.  Choose $\bar{z}$ such that $R(\bar{z}) = \bar{a}$. Let $\bar{y}'$ be a tuple of new symbols of the same length as $\bar{z}$. Consider the finite structure, in the signature of $K$, defined on the elements $\bar{y},\bar{y}',\bar{z}$ as follows. The relations $E$ and $C_i$ are defined on $\bar{y}$ and $\bar{z}$ as in $\mathcal{N}$. We set $y_i' E z_i$, and the equivalence classes are completely determined by this. Define $y_i C_k y'_j$ if and only if $x_i C_k x'_j$ (and $y'_i C_k y'_j$ if and only if $x'_i C_k x'_j$). There are no monochromatic triangles among 
$\bar{y},\bar{y}'$, or among $\bar{y},\bar{z}$. Since we have only used finitely many colors so far, we can color the remaining pairs so that there are no monochromatic triangles. 

The finite structure we have defined is in the class $K$, so we can find a realization of $\bar{y}'$ in $\mathcal{N}$. Then $R(\bar{y}') = \bar{a}$, so that $R(\bar{x}),R(\bar{x}') \sim^\beta R(\bar{y}),R(\bar{y}')$. Also, $\bar{x},\bar{x}' \equiv_{at} \bar{y},\bar{y}'$. Thus, 
$\bar{x},\bar{x}'\sim^\beta\bar{y},\bar{y}'$ by the inductive hypothesis (or, for $\beta = 0$, because $\bar{x},\bar{x}'\sim^\beta\bar{y},\bar{y}'$.  So, we have shown that $\bar{x}\sim^\alpha\bar{y}$.  On the other hand, if $\bar{x}\not\equiv_{at}\bar{y}$, then it is immediate that 
$\bar{x}\nsim^\alpha \bar{y}$. If $R(\bar{x})\nsim^\alpha R(\bar{y})$, then it is not hard to see that $\bar{x}\nsim^\alpha\bar{y}$.  
\end{proof}

\noindent
\textbf{Claim 3}. $SR(\mathcal{N})= \omega_1^{CK} +1$. 

\begin{proof} [Proof of Claim 3]

Let $x \in\mathcal{N}$ be such that $R(x)$ is in the ill-founded part of $\omega_1^{CK}(1+\eta)$. We claim that $SR(x) = \omega_1^{CK}$. Fix $\alpha < \omega_1^{CK}$. Let $y$ be such that $pred(R(y))$ is well-founded and $R(x)\sim^\alpha R(y)$. Now, there is no automorphism of the Harrison ordering taking $R(y)$ to $R(x)$, so there is no automorphism of $\mathcal{N}$ taking $y$ to $x$. Thus, $x$ and $y$ are in different automorphism orbits. Since $x$ and $y$ are singletons, $x \equiv_{at} y$.  Thus, by the previous claim, $x \sim^\alpha y$. Since $\alpha$ was arbitrary, $SR(x) = \omega_1^{CK}$, completing the proof of the claim.
\end{proof}

\noindent
\textbf{Claim 4}. $\mathcal{N}$ has no indiscernible ordered triple.

\begin{proof} [Proof of Claim 4]

It suffices to show that no three singleton elements are order indiscernible. Given $x$, $y$, and $z$, let $i$ be such that $x C_i y$. Since 
$\mathcal{N}$ has no monochromatic triangles, it cannot be the case that $y C_i z$ and $x C_i z$. Thus $x$, $y$, and $z$ are not indiscernible. 
\end{proof}

This completes the proof of Theorem \ref{thm:indis-triple}.
\end{proof}

Note that this construction is, in some sense, cheating. The structure $\mathcal{N}$ is effectively bi-interpretable (see \cite{HMMM}) with the Harrison ordering: the Harrison ordering lives inside $\mathcal{N}$ as the definable quotient modulo the definable equivalence relation $E$. The indiscernible sequence of the Harrison ordering becomes an indiscernible sequence of imaginaries in $\mathcal{N}$.

\begin{defn}

Fix a structure $\mathcal{A}$. An \emph{indiscernible sequence of imaginaries of $\mathcal{A}$} is a sequence $(E_i)_{i \in \omega}$ of equivalence classes of $\mathcal{A}$, modulo some $L_{\omega_1 \omega}$-definable equivalence relation, such that for any two finite subsequences $E_{i_1},\ldots,E_{i_n}$ and $E_{j_1},\ldots,E_{j_n}$ (with $i_1 < i_2 < \dots < i_n$ and $j_1 < j_2 < \dots < j_n$) there is an automorphism of $\mathcal{A}$ mapping $E_{i_k}$ to $E_{j_k}$. 
\end{defn}

\begin{prop}

Let $\mathcal{N}$ be the structure from Theorem \ref{thm:indis-triple}. Then $\mathcal{N}$ has an indiscernible sequence of imaginaries.

\end{prop} 

\begin{proof}

The map $R$ from above induced a bijection between the $E$-equivalence classes and the elements of $\omega_1^{CK}(1 + \eta)$. We claim that each automorphism of $\omega_1^{CK}(1 + \eta)$ induces an automorphism of $\mathcal{N}$. Then, since $\omega_1^{CK}(1 + \eta)$ has an indiscernible sequence, $\mathcal{N}$ will have an indiscernible sequence of imaginaries, namely the $E$-equivalence classes in bijection to the indiscernible sequence of $\omega_1^{CK}(1 + \eta)$.

It suffices to see that in the Fra\"{i}ss\'{e} limit $\mathcal{M}$, any permutation $\pi$ of the $E$-equivalence classes extends to an automorphism of $\mathcal{M}$. Let 
$\bar{a}$ and $\bar{b}$ be tuples of elements of $\mathcal{M}$ of the same length, satisfying the same atomic formulas, and such that if $a_i$ is in the $j$th equivalence class then $b_i$ is in the $\pi(j)$th equivalence class. Let $c$ be an additional element of $\mathcal{M}$. We can find an element $d$ of $\mathcal{M}$ such that $d$ is in the $\pi(j)$th equivalence class (if $c$ was in the $j$th equivalence class) and such that $\bar{b},d$ is colored in the same way as $\bar{a},c$. We can do this since $\bar{a},c$ has no monochromatic triangles. This lets us construct the desired automorphism using a back-and-forth construction.
\end{proof}

A construction similar to that of Theorem \ref{thm:indis-triple} allows us to turn any structure $\mathcal{M}$ into a structure $\mathcal{M}^*$ that is effectively bi-interpretable with $\mathcal{M}$, but has no indiscernible triples. Two structures which are effectively bi-interpretable have many of the same computability-theoretic properties; for example, they have the same computable dimension (see \cite{HKSS}). In light of this, we want not just a structure $\mathcal{M}$ with Scott rank $\omega_1^{CK} + 1$ and no indiscernible sequence, but a structure $\mathcal{M}$ with Scott rank $\omega_1^{CK} + 1$ and no indiscernible sequence of imaginaries. To produce such a structure, we use a construction originally due to Makkai \cite{M}, and refined by Knight and Millar \cite{KM} (see also \cite{K}).  Makkai used this construction to produce an arithmetical structure of Scott rank 
$\omega_1^{CK}$, and Knight and Millar used it to give the first example of a computable structure of Scott rank $\omega_1^{CK}$.  

Let $T \subseteq \omega^{< \omega}$ be a tree. We will define a new structure $\mathcal{A}(T)$. Let $T_n$ be the set of nodes at the $n$th level of $T$. For each $n$, we define $G_n = \mathcal{P}_{< \omega}(T_n)$ to be the collection of finite subsets of $T_n$.  Now, $G_n$ forms an abelian group under symmetric difference $\Delta$. The identity element of $G_n$ is the empty set, which we denote by $\id_n$. Let $G = \bigcup_n G_n$. The tree structure on $T$ induces a tree structure on $G$, which we will define using a predecessor relation $p$. Given $a \in G_{n+1}$, write $a = \{t_1,\ldots,t_n\}$. Then set $p(a)$ to be the sum of the predecessors of $t_1,\ldots,t_n$.  An element $t^*$ is in $p(a)$ if and only if the number of successors of $t^*$ in $a$ is odd. We have $p(\id_{n+1}) = \id_n$.  Note that $p$ is a homomorphism from $G_{n+1}$ to $G_n$. For each $a \in G$, we will define a unary function $f_a$. If $a \in G_m$, and $b \in G_n$, let $k = \min(m,n)$. Let $a^*$ and $b^*$ be the $p$-predecessors of $a$ and $b$ that are in $G_k$: $a^* = p^{m-k}(a)$ and $b^* = p^{n-k}(b)$. Then set $f_a(b) = a^* \Delta b^*$, noting that $f_a(b) = f_b(a)$. Let $\mathcal{A}(T)$ be the structure $(T,(f_a)_{a \in G})$.

We note the following facts from \cite{KM}. Since $a \in G_n$ if and only if $f_{\id_n}(a) = a$ and for all $m < n$, $f_{\id_m}(a) \neq a$, $G_n$ is preserved under automorphisms of $\mathcal{A}(T)$. Also, $p$ is preserved under automorphisms, as for $a \in G_{n+1}$ and $b \in G_n$, $p(a) = b$ if and only if $f_{\id_n}(a) = b$. Finally, for any $a \in G_n$, $a = f_a(\id_n)$ and so any automorphism of $\mathcal{A}(T)$ is determined by the images of the elements $\id_n$.

\begin{lem}[Lemma 3.3 of \cite{KM}]\label{lem-3-3}

Let $a \in G_n$, with $a \neq \id_n$. Then the tree rank of $a$ is the minimum of the tree ranks of $t$ for $t \in a$.

\end{lem}

\begin{lem}[Lemma 3.6 of \cite{KM}]\label{lem-3-6}

For $a \in G_n$, $a \equiv^\beta \id_n$ if and only if the tree rank of $a$ is at least $\omega \cdot \beta$.

\end{lem}

\begin{thm}[Theorem 3.7 and Lemma 4.3 of \cite{KM}]\label{thm:3-7}
There is a computable tree $T$ such that $SR(\mathcal{A}(T)) = \omega_1^{CK}$.
\end{thm}

\begin{lem}\label{lem:no-ind}
Let $T$ be a tree. Then $\mathcal{A}(T)$ does not have an indiscernible ordered triple of imaginaries.
\end{lem}

\begin{proof}

Suppose to the contrary that there is a definable equivalence relation with three equivalence classes $E_1$, $E_2$, and $E_3$ that form an indiscernible triple. Fix $a \in E_1$. Let $n$ be such that $a \in G_n$. There are automorphisms of $\mathcal{A}(T)$, one taking $E_1$ to $E_2$, and one taking $E_1$ to $E_3$. So, $E_2$ and $E_3$ both contain elements of $G_n$. Let $g$ be an automorphism of $\mathcal{A}(T)$ fixing $E_1$. Then for each $b \in G_n$, $b = f_{b \Delta a}(a)$, and so $g(b) = f_{b \Delta a}(g(a))$. Hence, the action of $g$ on $G_n$ is entirely determined by where $g$ sends $a$. Thus, there cannot be two automorphisms of $\mathcal{A}(T)$, one fixing $E_1$, $E_2$, and $E_3$ and the other fixing $E_1$ and mapping $E_2$ to $E_3$.  This contradicts the order-indiscerniblity of $E_1$, $E_2$, and $E_3$.
\end{proof}

\begin{thm}\label{thm:ind-plus}
There is a computable tree $T$ such that $\mathcal{A}(T)$ has Scott rank $\omega_1^{CK} + 1$.
\end{thm}

\begin{proof}

Let $T$ be the tree of finite decreasing sequences in the Harrison ordering. We claim that $\mathcal{A}(T)$ has Scott rank $\omega_1^{CK} + 1$. Note that at each level of $T$, there are elements of every computable tree rank, and there are elements with infinite tree rank. Let $G$ be the tree defined from $T$ as above. Then by Lemma \ref{lem-3-3}, at each level of $G$, there are elements of every computable tree rank, and there are elements with infinite tree rank. Fix $n$. Given $\beta < \omega_1^{CK}$, there is $a \in G_n$ with computable tree rank at least $\omega \cdot \beta$. By Lemma \ref{lem-3-6}, $a \equiv^\beta \id_n$, but $a \nequiv^\gamma \id_n$ for some $\gamma > \beta$. Hence $SR(\id_n) = \omega_1^{CK}$. It follows that $SR(\mathcal{A}(T)) = \omega_1^{CK} + 1$.
\end{proof}

\begin{cor}
There are computable structures $M$ of Scott rank $\omega_1^{CK}$, and of $\omega_1^{CK}+1$, that have no indiscernible ordered triples of imaginaries.
\end{cor}

\begin{proof}
This follows from Lemma \ref{lem:no-ind} and Theorems \ref{thm:3-7} and \ref{thm:ind-plus}.
\end{proof}

We end with an open question.

\medskip

\noindent
\textbf{Question}.  Is there a structure of Scott rank $\omega_1^{CK}$ that is computably approximable and has no indiscernible sequences of imaginaries?


\begin{thebibliography}{999}

\bibitem{CGKM}  W.\ Calvert, S.\ S.\ Goncharov, J.\ F.\ Knight, and J.\ Millar, ``Categoricity of computable infinitary theories'', \emph{Arch.\ Math.\ Logic}, vol.\ 48(2009), 
pp.\ 25-38.

\bibitem{CKM}  W.\ Calvert, J.\ F.\ Knight, and J.\ Millar, ``Computable trees of Scott rank $\omega_1^{CK}$, and computable approximation'',  \emph{J.\ Symb.\ Logic}, vol.\ 71(2006), pp.\ 283-298.

\bibitem{CHMM}  B.\ Csima, V.\ Harizanov, R.\ Miller, and A.\ Montalb\'{a}n, ``Computability of Fra\"{i}ss\'{e} limits'', \emph{J.\ Symb.\ Logic}, vol.\ 76(2011), pp.\ 66-93.  

\bibitem{F}  C.\ Freer, \emph{Models with High Scott Rank}. Ph.D. thesis, Harvard University, 2008. 

\bibitem{Harrison}  J.\ Harrison, ``Recursive pseudo well-orderings'', \emph{Trans.\ of the Amer.\ Math.\ Soc.}, vol.\ 131(1968), pp.\ 526-543.

\bibitem{HMMM} M. Harrison-Trainor, A. Melnikov, R. Miller, and A. Montalb\'an, ``Computable Functors and Effective Interpretability'', preprint.

\bibitem{H}  C.\ W.\ Henson, ``A family of countable homogeneous graphs'', \emph{Pacific J.\ Math.}, vol.\ 38(1971), pp.\ 69-83.

\bibitem{HKSS} D.\ Hirschfeldt, B.\ Khoussainov, R.\ Shore, and A.\ Slinko, ``Degree spectra and computable dimensions in algebraic structures'', \emph{APAL}, vol.\ 115(2002), pp.\ 71-113. 

\bibitem{K}  J.\ F.\ Knight ``Structures of Scott rank $\omega_1^{CK}$'', expository paper, in \emph{Models, Logics, and Higher-Dimensional Categories:  A Tribute to the Work of Mih\'{a}ly Makkai}, ed.\ by Hart, et al, 2011, CRM Proceedings and Lecture Notes, pp.\ 157-168.

\bibitem{KM}  J.\ F.\ Knight and J.\ Millar, ``Computable structures of rank $\omega_1^{CK}$'',  \emph{J.\ Math.\ Logic}, vol.\ 10(2011), pp.\ 31-43. 

\bibitem{M}  M.\ Makkai, ``An example concerning Scott heights'', \emph{J. of Symb.\ Logic}, vol.\ 46(1981), pp.\ 301-318.

\bibitem{MS}  J.\ Millar and G.\ E.\ Sacks, ``Atomic models higher up'', \emph{APAL}, vol.\ 155(2008), pp.\ 225-241. 

\bibitem{N}  M.\ Nadel, ``Scott sentences and admissible sets'', \emph{Annals of Math.\ Logic}, vol.\ 7(1974), pp.\ 267-294.

\bibitem{Sc}  D.\ Scott, ``Logic with denumerably long formulas and finite strings of quantifiers'', \emph{The Theory of Models}, ed.\ by J.\ Addison, L.\ Henkin, and A.\ Tarski, North-Holland, 1965, pp.\ 329-341.                                                                                                                                                                                                                                                                                                         

\end{thebibliography}
\end{document}